\documentclass[a4paper, reqno, english,11pt]{article}

\usepackage{amsfonts, amsthm, amsmath, amssymb}

\RequirePackage{mathrsfs} \let\mathcal\mathscr

\numberwithin{equation}{section}

\renewcommand{\phi}{\varphi}

\newtheorem{theorem}{Theorem}
\newtheorem*{thm*}{Theorem}
\newtheorem{lemma}{Lemma}

\renewcommand{\le}{\leqslant}

\renewcommand{\ge}{\geqslant}

\theoremstyle{definition}

\newcommand{\dif}{\mathrm{d}}

\begin{document}

\title{Exponential sums with Dirichlet coefficients of $L$-functions}
\author{Stephan Baier}
\maketitle

\begin{abstract}
Improving and extending recent results of the author \cite{Bai}, we conditionally estimate exponential sums with Dirichlet coefficients of $L$-functions, both over all integers and over all primes in an interval. In particular, we establish new conditional results on exponential sums with Hecke eigenvalues and squares of Hecke eigenvalues over primes. We employ these estimates to improve our recent result in \cite{Bai} on squares of Hecke eigenvalues at Piatetski-Shapiro primes under the Riemann Hypothesis for symmetric square $L$-functions for Hecke eigenforms. 
\end{abstract}



\section{Introduction} \label{intro}
In this paper, by making several refinements, we improve our conditional estimate in \cite{Bai} for exponential sums of the form
\begin{equation} \label{fullsum}
\sum\limits_{n\sim N} a_ne(f(n)), 
\end{equation}
where $a_n$ is the $n$-th coefficient of a Dirichlet series $F(s)$ in a certain extension of the Selberg class (see section \ref{assump}). In particular, we obtain a conditional improvement of a result of Jutila on exponential sums with Hecke eigenvalues. We also derive a conditional estimate for the same sums as in \eqref{fullsum} but restricted to primes, i.e.  
\begin{equation} \label{primesum}
\sum\limits_{p\sim N} a_pe(f(p)),
\end{equation}
which we then use to improve our recent result in \cite{Bai} on squares of Hecke eigenvalues at Piatetski-Shapiro primes.  In \cite{Bai}, we employed Vaughan's identity to relate exponential sums over primes to exponential sums over all integers in intervals. In this paper, we shall use a more direct approach to handle the prime condition, relating the said sums with logarithms of twists of $F(s)$ with Dirichlet characters. 

In the next section, we shall state the assumptions on $F(s)$ and $f(x)$ under which we shall establish our estimates. The following notations will be used throughout this paper.\\ \\
{\bf Notations:}\medskip\\
1) By $\varepsilon$ and $\eta$, we denote arbitrarily small positive numbers.\\
2) The symbol $p$ is reserved for primes, and $\mathbb{P}$ denotes the set of primes.\\
3) The symbol $s$ is reserved for complex numbers, and we write $s=\sigma+it$, $\sigma$ being the real part and $t$ being the imaginary part of $s$.

\section{Assumptions} \label{assump}

{\bf Conditions on $f(x)$:}\medskip\\
We assume that $f:[1,\infty)\rightarrow \mathbb{R}^+$ satisfies the following conditions i)-vii).\medskip\\
i) $f$ is four times continuously differentiable.\\
ii) $f$ is monotonically increasing.\\
iii) $f(x)\asymp f(2x)$ for all $x\ge 1$.\\
iv) $f^{(k)}(x) \asymp f(x)/x^k \mbox{ for all } x\ge 1 \mbox{ and } k=1,2,3,4$.\\
v) $f'(x)+xf''(x)\asymp f(x)/x \mbox{ for all } x\ge 1$.\\
vi) $2f''(x)+xf'''(x)\asymp f(x)/x^2 \mbox{ for all } x\ge 1$.\\
vii) $2f''(x)-xf'''(x)\asymp f(x)/x^2 \mbox{ for all } x\ge 1$.\\ \\
{\bf General conditions on $F(s)$:}\medskip\\
We assume that 
$$
F(s)=\sum_{n=1}^\infty a_nn^{-s}
$$
is a Dirichlet series, absolutely convergent for $\Re s>1$, which satisfies the following conditions a) and b).\medskip\\
a) $F(s)$ has the following properties.\medskip

(I) ({\it Analiticity}) There exists some $m\in \mathbb{N}$ such that $(s-1)^mF(s)$ extends to a holomorphic function of finite order on the half plane $\Re s\ge 1/2$.

(II) ({\it Ramanujan hypothesis}) $a_1=1$ and $a_n \ll_\varepsilon n^\varepsilon$ for any $\varepsilon>0$.

(III) ({\it Euler product}) For $\Re s>1$, the function $F(s)$ can be written as a product over primes in the form
\begin{equation} \label{Euler}
F(s)=\prod_p F_p(s),
\end{equation}
where  $\log F_p(s)$ is a Dirichlet series of the form
\begin{equation} \label{logdev}
\log F_p(s)=\sum_{k=1}^\infty b_{p^k}p^{-ks}
\end{equation}
with complex coefficients $b_{p^k}$ satisfying
\begin{equation} \label{coeffestimate}
b_{p^k}=O(p^{k\theta})
\end{equation}
for some $\theta<1/2$.\\ 
b) For any Dirichlet character $\chi$ define 
\begin{equation} \label{Fschi}
F(s,\chi):=\sum\limits_{n=1}^{\infty} a_n \chi(n) n^{-s} \quad \mbox{for } \Re s>1.
\end{equation}
Then  $(s-1)^mF(s,\chi)$ extends to an entire function again.\\ \\
{\bf Special conditions on $F(s)$:}\medskip\\
For the estimation of the exponential sums in \eqref{fullsum}, we shall assume the following.\medskip\\
c) The family of functions $F(s,\chi)$ satisfies the Lindel\"of Hypothesis in the $t$- and $q$-aspects, i.e.
$$
F\left(\frac{1}{2}+it,\chi\right)\ll |tq|^{\varepsilon} \quad \mbox{for all $|t|\ge 1$, $q\in \mathbb{N}$} \mbox{ and characters $\chi \bmod{q}$}.
$$
For the estimation of the exponential sums over primes in \eqref{primesum}, we shall make somewhat stronger condtions.\medskip\\
c') The family of functions $F(s,\chi)$ satisfies the Riemann Hypothesis, i.e. all zeros with real part $\ge 1/2$ are actually located at the critical line $\Re s=1/2$. Moreover, for every $\varepsilon>0$ there exist $A,B\in \mathbb{R}$, $C,D>0$ such that 
\begin{equation} \label{logbound}
C(q|t|)^A\le |F\left(s,\chi\right)| \le D(q|t|)^{B}  \quad \mbox{for all $\sigma\ge 1/2+\varepsilon$, $|t|\ge 1$, $q\in\mathbb{N}$ and characters $\chi \bmod{q}$.}
\end{equation}
d) Every $\theta>0$ in \eqref{coeffestimate} is admissible.

\section{Results}
Our estimates for exponential sums are as follows.

\begin{theorem} \label{expsum} 
Suppose the conditions {\rm i)-vii)} on $f(x)$ and {\rm a), b), c)} on $F(s)$ in section {\rm \ref{assump}} are satisfied. Fix $\eta>0$. Suppose that $1\le N<N'\le 2N$ and
\begin{equation} \label{anotherfcond1}
N^{3/4+\eta}\le f(N)\le N^{3/2-\eta}.
\end{equation} 
Then 
$$
\sum\limits_{N<n\le N'} a_n e(f(n))\ll_{f,\eta,\varepsilon} N^{3/4+\varepsilon}f(N)^{1/6}.
$$
\end{theorem}

\begin{theorem} \label{expsumprimes}
Suppose the conditions {\rm i)-vii)} on $f(x)$ and {\rm a), b), c'), d)} on $F(s)$ in section {\rm \ref{assump}} are satisfied. Fix $\eta>0$. Suppose that $1\le N<N'\le 2N$ and
\eqref{anotherfcond1} holds. Then 
$$
\sum\limits_{N<p\le N'} a_p e(f(p))\ll_{f,\eta,\varepsilon} N^{3/4+\varepsilon}f(N)^{1/6}.
$$
\end{theorem}

From Theorem \ref{expsumprimes}, we shall deduce the following general result on the behaviour of the coefficients $a_p$ at Piatetski-Shapiro primes $p$. 

\begin{theorem} \label{piatetski} Suppose the conditions {\rm a), b), c'), d)} on $F(s)$ in section {\rm \ref{assump}} are satisfied. Let $1<c< 12/11$ be fixed. Then there exists $\delta>0$ such that
$$
\sum\limits_{\substack{n\le N\\ \left[n^c\right]\in \mathbb{P}}} a_{\left[n^c\right]} = \sum\limits_{p\le N^c} \left((p+1)^{1/c}-p^{1/c}\right)a_p + O\left(N^{1-\delta}\right)
\quad \mbox{as } N\rightarrow \infty.
$$
\end{theorem}

Let now $G$ be a Hecke eigenform of weight $\kappa$ for the full modular group SL$_2(\mathbb{Z})$. By $\lambda(n)$ we denote the normalized $n$-th Fourier coefficient of 
$G$, {\it i.e.}
$$
G(z)=\sum\limits_{n=1}^{\infty} \lambda(n)n^{(\kappa-1)/2}e(nz) \quad \mbox{for } \Im z>0, \mbox{ and } \lambda(1)=1.  
$$
Let $L(G,s)$ be the $L$-function for $G$, defined by
$$
L(G,s):= \sum\limits_{n=1}^{\infty} \lambda\left(n\right)n^{-s} \quad \mbox{for } \Re s>1.
$$
We note that $L(G,s)$ can be written in the form
\begin{equation} \label{infprod}
L(G,s):=\prod\limits_p \left(1-\alpha_p p^{-s}\right)^{-1}\left(1-\overline{\alpha_p}p^{-s}\right)^{-1} \quad \mbox{for } \Re s>1,
\end{equation}
where $\alpha_p$ are complex numbers with $\left|\alpha_p \right|=1$ (see \cite{Iwaniec}, section 13.8., for example).
More generally, for any Dirichlet character $\chi$ let $L(G\otimes \chi,s)$ be the $L$-function for $G$ twisted with $\chi$, defined by
$$
L(G \otimes \chi,s)=\sum\limits_{n=1}^{\infty} \chi(n) \lambda\left(n\right)n^{-s} \quad \mbox{for } \Re s>1.
$$
Furthermore, let $L(\mbox{Sym}^2\ G,s)$ be the symmetric square $L$-function for $G$, defined by
$$
L(\mbox{Sym}^2\ G,s)=
\zeta\left(2s\right)\sum\limits_{n=1}^{\infty} \lambda\left(n^2\right)n^{-s}  \quad \mbox{for } \Re s>1.
$$
We note that (see \cite{Iwaniec}, section 13.8.)
\begin{equation} \label{infprod1}
L\left(\mbox{Sym}^2\ G,s\right)=\prod\limits_p \left(1-\alpha_p^2p^{-s}\right)^{-1}\left(1-p^{-s}\right)^{-1}\left(1-\overline{\alpha_p}^{2}p^{-s}\right)^{-1}  \quad \mbox{for } \Re s>1,
\end{equation}
where the complex numbers $\alpha_p$ are as given in \eqref{infprod}.
More generally, let $L(\mbox{Sym}^2\ G\otimes \chi,s)$ be the symmetric square $L$-function for $G$ twisted with $\chi$, defined by
$$
L\left(\mbox{Sym}^2\ G\otimes \chi,s\right)=
L\left(2s,\chi^2\right)\sum\limits_{n=1}^{\infty} \chi(n) \lambda\left(n^2\right)n^{-s} \quad \mbox{for } \Re s>1.
$$
We note that if $F(s)=L(\mbox{Sym}^2\ G,s)$, then $F(s,\chi)=L(\mbox{Sym}^2\ G \otimes \chi,s)$, in the sense of \eqref{Fschi}. \\ \\
{\bf Remark 1:} It is well-known that $F(s)=L(G,s)$ satisfies the conditions a) and b) in section 2.  The same is true for $F(s)=L(\mbox{Sym}^2\ G,s)$ by work of Shimura \cite{shim}. Moreover, if $\chi$ is a primitive character, then $L(G\otimes \chi,s)$ and $L(\mbox{Sym}^2\ G\otimes \chi,s)$ are $L$-functions in the sense of section 5.1. in \cite{IwKo}, and the analytic conductors of them depend polynomially on $s$ and the conductor of $\chi$. Under the Riemann Hypothesis for these $L$-functions, it therefore follows from Theorem 5.19. in \cite{IwKo} that they satisfy the bound \eqref{logbound}. Furthermore, it can be easily deduced from \eqref{infprod}, \eqref{infprod1} and $|\alpha_p|=1$ that $L(G,s)$ and $L(\mbox{Sym}^2\ G,s)$ satisfy condition d) in section 2.
Finally, it is a consequence of Theorem 1 in \cite{ConGosh} that the Riemann Hypothesis for all $L$-functions $L(G\otimes \chi,s)$ and $L(\mbox{Sym}^2\ G\otimes \chi,s)$ implies that they satisfy the  Lindel\"of Hypothesis in the $t$- and $q$-aspects, i.e.
\begin{equation} \label{lindG}
L\left(G \otimes \chi,\frac{1}{2}+it\right) \ll |qt|^{\varepsilon} \quad \mbox{for all $|t|\ge 1$, $q\in \mathbb{N}$ and characters $\chi \bmod{q}$}
\end{equation}
resp.
\begin{equation} \label{lindG1}
L\left(\mbox{Sym}^2\ G \otimes \chi,\frac{1}{2}+it\right) \ll |qt|^{\varepsilon} \quad \mbox{for all $|t|\ge 1$, $q\in \mathbb{N}$ and characters $\chi \bmod{q}$}.
\end{equation}

The following result on exponential sums with Hecke eigenvalues and squares of Hecke eigenvalues over primes follows from Theorem 2, Remark 1 and the fact that the $p$-th Dirichlet coefficient $a_p$ of $L(\mbox{Sym}^2 G,s)$ coincides with $\lambda(p^2)$.

\begin{theorem} \label{heckeprimes} Suppose the conditions {\rm i)-vii)} on $f(x)$ in section {\rm \ref{assump}} are satisfied. Then we have the following.\medskip\\
{\rm (i)} If the Riemann Hypothesis holds for all $L$-functions $L(G \otimes \chi,s)$, then
$$
\sum\limits_{N<p\le N'} \lambda(p) e(f(p))\ll_{f,\eta,\varepsilon} N^{3/4+\varepsilon}f(N)^{1/6}.
$$
{\rm (ii)} If the Riemann Hypothesis holds for all $L$-functions $L\left(\mbox{Sym}^2\ G\otimes \chi,s\right)$, then
$$
\sum\limits_{N<p\le N'} \lambda\left(p^2\right) e(f(p))\ll_{f,\eta,\varepsilon} N^{3/4+\varepsilon}f(N)^{1/6}.
$$
\end{theorem}

Finally, we shall derive the following new result on squares of Hecke eigenvalues at Piatetski-Shapiro primes from Theorem \ref{piatetski}.

\begin{theorem} \label{heckesquares} Let $1<c< 12/11$ be fixed. Assume that the Riemann Hypothesis holds for all $L$-functions
$L\left(\mbox{Sym}^2\ G \otimes \chi,s\right)$. Then we have
$$
\sum\limits_{\substack{n\le N\\ \left[n^c\right]\in \mathbb{P}}} \lambda\left(\left[n^c\right]\right)^2\sim \frac{N}{c\log N} \quad
\mbox{as } N\rightarrow \infty.
$$ 
\end{theorem}

In \cite{Bai}, we proved Theorem \ref{heckesquares} for $c$ in the smaller range $1<c<25/24$ under the Lindel\"of Hypothesis \eqref{lindG1}.

\section{Approximation of the amplitude function in short intervals} \label{Approxsection}
Our goal is to establish non-trivial bounds for exponential sums of the form
$$
\sum\limits_{n\sim N} c_n e(f(n)),
$$
where $c_n=a_n$ in the case of Theorem 1, and $c_n=a_n$ if $n$ is prime and $c_n=0$ otherwise in the case of Theorem 2. 
To this end, we split this exponential sum into short subsums in the same way as in \cite{Bai}. To keep this paper self-contained, we copy this treatment from \cite{Bai}.  

For $x\ge 1$, let 
\begin{equation} \label{hdef}
h(x):=f'(x)+xf''(x).
\end{equation}
By the condition vi) on $f$ in section 1, we have 
\begin{equation} \label{hprime}
h'(x)=2f''(x)+xf'''(x)\asymp \frac{f(x)}{x^2}.
\end{equation}
Since $f$ is assumed to take only positive values and $h'$ is continuous, it follows that $h'(x)$ doesn't change sign and hence, $h(x)$ is monotonically increasing or decreasing. In the sequel, we assume without loss of generality that $h(x)$ is 
monotonically decreasing. Let $Q$ be a real parameter satisfying
\begin{equation} \label{Qcond}
\frac{N^{1+\eta}}{f(N)}\le Q\le N,
\end{equation}
to be chosen later.  We make a Farey dissection of level $Q$ of the interval $[h(N'),h(N))$, i.e., we write $[h(N'),h(N))$ as the disjoint union of intervals of the form 
$$
\left[\left.\frac{l}{q}-\frac{M_1}{qQ},\frac{l}{q}+\frac{M_2}{qQ}\right)\right. \cap [h(N'),h(N)),
$$
where $M_1,M_2\asymp 1$, $q\le Q$ and $(q,l)=1$. Projecting these intervals back into $(N,N']$ under the map $h^{-1}$, we get intervals of the form
$$
h^{-1}\left(\left.\left[\frac{l}{q}-\frac{M_1}{qQ},\frac{l}{q}+\frac{M_2}{qQ}\right.\right) \cap [h(N'),h(N))\right)=
(x_0-m_1,x_0+m_2] \subseteq (N,N']
$$
with 
$$
x_0=h^{-1}\left(\frac{l}{q}\right)
$$
and
\begin{equation} \label{m1m2sizes}
m_1,m_2 \ll \frac{1}{qQ} \cdot \left(h^{-1}\right)'\left(\frac{l}{q}\right)=\frac{1}{qQ}\cdot \frac{1}{h'(x_0)} \asymp 
\frac{N^2}{qQf(N)},
\end{equation}
by \eqref{hprime} and the conditions ii) and iii) on  $f$ in section 1. In the following, we write 
$$
I(x_0):=(x_0-m_1,x_0+m_2].
$$

Now our aim is to produce a non-trivial estimate for the subsum over the interval $I(x_0)$. To this end, as in \cite{Bai}, we approximate the amplitude function $f(x)$ by the function
\begin{equation} \label{gdef}
g(x)=h(x_0)x-x_0^2f''(x_0)\log x+C=\frac{l}{q}\cdot x-x_0^2f''(x_0)\log x+C,
\end{equation}
where 
$$
C:=f(x_0)-h(x_0)x_0+x_0^2f''(x_0)\log x_0.
$$
The idea behind this approximation is that $e(g(x))$ can be written as a constant times an additive character mod $q$ times a complex power of $n$, which shall turn out to be useful.
Moreover, by the definition of $h(x)$ in \eqref{hdef}, we have
$$
g(x_0)=f(x_0),\quad g'(x_0)=f'(x_0), \quad \mbox{and} \quad
g''(x_0)=f''(x_0),
$$
which ensures a good approximation of $f(x)$ by $g(x)$ near $x_0$.
More precisely, using Taylor's theorem, we deduce that 
\begin{equation} \label{approxi1}
\begin{split}
(f-g)'(x)=& \frac{1}{2}\cdot \left(x-x_0\right)^2 (f-g)'''(x_0)+O\left((x-x_0)^3(f-g)^{(4)}(c_1)\right),\\
(f-g)''(x)=& \left(x-x_0\right)(f-g)'''(x_0)+O\left((x-x_0)^2(f-g)^{(4)}(c_2)\right),\\
(f-g)'''(x)= & (f-g)'''(x_0)+O\left((x-x_0)(f-g)^{(4)}(c_3)\right)
\end{split}
\end{equation}
for suitable $c_1,c_2,c_3\in I(x_0)$ depending on $x$. We further note that, by the conditions ii), iii), iv) and vii) on $f$, we have
\begin{equation} \label{third}
(f-g)'''(x_0)=f'''(x_0)-\frac{2f''(x_0)}{x_0}\asymp \frac{f(N)}{N^3}
\end{equation}
and
\begin{equation}  \label{fourth}
(f-g)^{(4)}(c)=f^{(4)}(c)+\frac{6x_0^2f''(x_0)}{c^4} \ll \frac{f(N)}{N^4} \quad \mbox{if } c\in I(x_0). 
\end{equation}
From \eqref{Qcond}, \eqref{m1m2sizes}, \eqref{approxi1}, \eqref{third} and \eqref{fourth}, it follows that
\begin{equation} \label{approxi2}
\begin{split}
(f-g)'(x) \ll & \frac{N}{(qQ)^2f(N)}\\
(f-g)''(x) \ll & \frac{1}{qQN}\\
(f-g)'''(x)\asymp & \frac{f(N)}{N^3} 
\end{split}
\end{equation}
if $x\in I(x_0)$. 

In the following, we write
\begin{equation}
\sum\limits_{n\in I(x_0)} c_n e(f(n)) = \sum\limits_{n\in I(x_0)} c_n e(g(n)) e(f(n)-g(n)). 
\end{equation}
In \cite{Bai}, we removed the slowly oscillating term $e(f(n)-g(n))$ using partial summation and then treated the sums
$$
\sum\limits_{n} c_n e(g(n))
$$ 
using Perron's formula. In this paper, we combine these two processes, which leads to sharper estimates.

\section{Applications of Perron's formula and partial summation}
We begin with a general treatment of expressions of the form
$$
\sum\limits_{x_1<n\le x_2} C_nZ(n),
$$
where $(C_n)$ is a sequence of complex numbers satisfying $C_n\ll n^{\varepsilon}$ for any given $\varepsilon>0$, $0<x_1<x_2$, and $Z:[x_1,x_2]\rightarrow \mathbb{C}$ is continuously differentiable. Later, we will focus on the special case when
\begin{equation} \label{choosing}
C_n=c_n e(g(n)), \quad x_1=x_0-m_1, \quad x_2=x_0+m_2, \quad Z(x)=e(f(x)-g(x)).
\end{equation}
Using partial summation, we have
\begin{equation} \label{partialsummation}
\sum\limits_{x_1<n\le x_2} C_nZ(n) = Z(x_2)\left(\sum\limits_{x_1<n\le x_2} C_n\right) - \int\limits_{x_1}^{x_2} \left(\sum\limits_{x_1<n\le u} C_n\right) Z'(u) \dif u. 
\end{equation}
From Perron's formula (see \cite{Bru}, Lemma 1.4.2., for example) and the condition $C_n\ll n^{\varepsilon}$, we have
\begin{equation} \label{perronsformula}
\sum\limits_{x_1<n\le u} C_n= \frac{1}{2\pi i} \int\limits_{1+\varepsilon-iT_0}^{1+\varepsilon+iT_0} \left(\sum\limits_{n=1}^{\infty} C_nn^{-s}\right)\left(u^s-x_1^s\right)\frac{\dif s}{s}+
O\left(\frac{u^{1+\varepsilon}}{T_0}+u^{\varepsilon}\right)
\end{equation}
for $T_0\ge 1$, to be fixed later. Plugging \eqref{perronsformula} into \eqref{partialsummation}, we get
\begin{equation} \label{start}
\begin{split}
\sum\limits_{x_1<n\le x_2} C_nZ(n) = \frac{1}{2\pi i} \int\limits_{1+\varepsilon-iT_0}^{1+\varepsilon+iT_0} H(s)\mathcal{I}(s)\frac{\dif s}{s}+
O\left(\left(\frac{x_2^{1+\varepsilon}}{T_0}+x_2^{\varepsilon}\right)\left(|Z(x_2)|+\int\limits_{x_1}^{x_2} |Z'(u)| \dif u\right)\right),
\end{split}
\end{equation}
where
\begin{equation} \label{settings}
H(s):=\sum\limits_{n=1}^{\infty} C_nn^{-s} \quad \mbox{and} \quad \mathcal{I}(s):=Z(x_2)(x_2^s-x_1^s)-\int\limits_{x_1}^{x_2} Z'(u)\left(u^s-x_1^s\right) \dif u.
\end{equation}

Now we assume that $H(s)$ extends to a holomorphic function on 
$$
\left\{ s\in \mathbb{C} \ :\ \Re s \ge \sigma_0, \ \ |\Im s|\le T_0 \right\}
$$
for some $\sigma_0<1$ which we will fix later. Then using Cauchy's theorem, it follows that
\begin{equation} \label{aftercauchy}
\begin{split}
\sum\limits_{x_1<n\le x_2} C_nZ(n) 
= & \frac{1}{2\pi i} \left(\int\limits_{1+\varepsilon-iT_0}^{\sigma_0-iT_0}+\int\limits_{\sigma_0-iT_0}^{\sigma_0+iT_0}+
\int\limits_{\sigma_0+iT_0}^{1+\varepsilon+iT_0}\right)  H(s)\mathcal{I}(s)\frac{\dif s}{s} \\ & +
O\left(\left(\frac{x_2^{1+\varepsilon}}{T_0}+x_2^{\varepsilon}\right)\left(|Z(x_2)|+\int\limits_{x_1}^{x_2} |Z'(u)| \dif u\right)\right).
\end{split}
\end{equation}

Next, we bound $H(s)$, $\mathcal{I}(s)$ and the $O$-term in the situation when the para\-meters and functions are given by \eqref{choosing}. We have 
$$
Z'(x)=2\pi i(f-g)'(x)\cdot e(f(x)-g(x)).
$$
By \eqref{m1m2sizes} and \eqref{approxi2}, we have
\begin{equation}
\int\limits_{x_1}^{x_2} |Z'(u)|\dif u \ll \frac{N^3}{q^3Q^3f(N)^2}.
\end{equation}
This together with \eqref{choosing} and $N\le x_1<x_2\le N'$ implies that the $O$-term is bounded by
\begin{equation} \label{Otermbound}
\left(\frac{x_2^{1+\varepsilon}}{T_0}+x_2^{\varepsilon}\right)\left(|Z(x_2)|+\int\limits_{x_1}^{x_2} |Z'(u)| \dif u\right) \ll N^{\varepsilon}\left(1+\frac{N}{T_0}\right)\left(1+\frac{N^3}{q^3Q^3f(N)^2}\right).
\end{equation}

\section{Estimation of Dirichlet series}
By \eqref{choosing} and the definition of $g(x)$ in \eqref{gdef}, if $\Re s>1$, we have
\begin{equation} \label{Hs}
H(s)=\sum\limits_{n=1}^{\infty} C_nn^{-s} =  e(C) \sum\limits_{n=1}^{\infty} c_n e\left(n\cdot \frac{l}{q}\right)n^{-s-iT},
\end{equation}
where 
\begin{equation} \label{Tchoice}
T:=2\pi x_0^2 f''(x_0).
\end{equation}
We recall that $(l,q)=1$ and note that 
\begin{equation} \label{Tsize}
T\asymp f(N)
\end{equation}
by condition (iv) on $f$.

Next, we divide the sum over $n$ according to greatest common divisors with $q$, i.e. we write
\begin{equation*}
\sum\limits_{n=1}^{\infty} c_n e\left(n\cdot \frac{l}{q}\right)n^{-s-iT}= \sum\limits_{d|q} \sum\limits_{(n,q)=d} c_n e\left(n\cdot \frac{l}{q}\right)n^{-s-iT}  =
\sum\limits_{d|q} d^{-s-iT} \sum\limits_{(n,q/d)=1} c_{dn} e\left(n\cdot \frac{l}{q/d}\right) n^{-s-iT}.
\end{equation*}
Further, for $(nl,q/d)=1$, we rewrite the additive character above as a linear combination of multiplicative Dirichlet characters in the form
\begin{equation*}
e\left(n\cdot \frac{l}{q/d}\right)=\frac{1}{\varphi(q/d)}\cdot \sum\limits_{\chi \bmod{q/d}} \chi(l)\tau(\overline{\chi})\chi(n),
\end{equation*}
where $\tau(\overline{\chi})$ is the Gauss sum associated with $\overline{\chi}$. It follows that
\begin{equation} \label{rewritten}
\sum\limits_{n=1}^{\infty} c_n e\left(n\cdot \frac{l}{q}\right)n^{-s-iT}=\sum\limits_{d|q} d^{-s-iT}\cdot \frac{1}{\varphi(q/d)} \cdot \sum\limits_{\chi \bmod{q/d}} \chi(l)\tau(\overline{\chi}) \cdot
\sum\limits_{n=1}^{\infty} c_{dn}\chi(n)n^{-s-iT}.
\end{equation}
We now treat the following cases: 1) $c_n=a_n$ for all $n\in \mathbb{N}$ and 2) $c_n=a_n$ if $n$ prime and $c_n=0$ otherwise. Here $a_n$ denotes the Dirichlet coefficients of the function
$$
F(s)=\sum\limits_{n=1}^{\infty} a_n n^{-s},
$$ 
introduced in section 2.

{\it Case 1}: $c_n=a_n$ for all $n\in \mathbb{N}$, and  the conditions in Theorem \ref{expsum} are satisfied. This case has been treated completely in \cite{Bai}. In the following, we recall the results in \cite{Bai}. First, we have the factorization
\begin{equation} \label{factori}
\sum\limits_{n=1}^{\infty} c_{dn}\chi(n)n^{-s}=\sum\limits_{n=1}^{\infty} a_{dn}\chi(n)n^{-s} = G_{d}(s,\chi)F(s,\chi),
\end{equation}
where $F(s,\chi)$ is defined as in section 2 and 
$$
G_d(s,\chi)=\prod\limits_{p|d} \frac{\sum\limits_{k=0}^{\infty} 
a_{p^{\alpha(p)+k}}\chi^k(p)p^{-ks}}{\sum\limits_{k=0}^{\infty} 
a_{p^k}\chi^k(p)p^{-ks}}
$$
if 
$$
d=\prod\limits_{p|d} p^{\alpha(p)}
$$
is the prime number factorization of $d$. Both $F(s,\chi)$ and $G_d(s,\chi)$ extend to holomorphic functions on 
$$
\{ s \ :\ \Re s\ge 1/2, \ \Im s\ge 1\}.
$$
Moreover,
\begin{equation} \label{GdF}
G_d(s,\chi)\ll d^{\varepsilon} \quad \mbox{and} \quad F\left(s,\chi\right) \ll |qt|^{\varepsilon} \quad \mbox{uniformly for } \sigma \ge 1/2 \mbox{ and } t \ge 1.
\end{equation}
From \eqref{Hs}, \eqref{rewritten}, \eqref{factori}, \eqref{GdF} and $|\tau(\chi)|\le \sqrt{q}$, we deduce that $H(s)$ extends to a holomorphic function on 
\begin{equation}
\mathcal{M}:=\{ s \ :\ \Re s\ge 1/2,\ |\Im s|\le T_0\}
\end{equation}
if 
\begin{equation} \label{T0cond}
0<T_0\le |T|-1, 
\end{equation}
and, under the same condition \eqref{T0cond}, we have
\begin{equation} \label{Hsbound1}
H(s)\ll q^{1/2+\varepsilon}|T|^{\varepsilon} \quad \mbox{uniformly in } \mathcal{M}.
\end{equation}
We note that by \eqref{Tsize}, \eqref{T0cond} holds if 
\begin{equation} \label{T0condi}
0<T_0\le f(N)N^{-\eta}
\end{equation}
for any fixed $\eta>0$ and $N$ large enough. 

{\it Case 2}: $c_n=a_n$ if $n$ prime and $c_n=0$ otherwise, and the conditions in Theorem \ref{expsumprimes} are satisfied. Then we have 
\begin{equation} \label{umrechnung}
\sum\limits_{n=1}^{\infty} c_{dn}\chi(n)n^{-s}=\begin{cases}  M(s,\chi)  & \mbox{if } d=1\\ a_d & \mbox{if } d \mbox{ is a prime}\\ 0 & \mbox{otherwise,} \end{cases}
\end{equation}
where
$$
M(s,\chi):=\sum\limits_{p} a_{p}\chi(p)p^{-s}.
$$

In the following, we focus on the first case when $d=1$, in which we approximate the series in question as follows. Using \eqref{Euler}, \eqref{logdev} and the complete multiplicativity of Dirichlet characters, we find
$$
F(s,\chi)=\sum\limits_{n=1}^{\infty} a_n\chi(n)n^{-s}=\prod\limits_p F_p(s,\chi) \quad \mbox{if } \Re s>1, 
$$
where 
$$
\log F_p(s,\chi)=\sum\limits_{k=1}^{\infty} b_{p^k}\chi^k(p)p^{-ks}.
$$
It follows that
$$
\log F(s,\chi)= \sum\limits_{p} \sum\limits_{k=1}^{\infty} b_{p^k}\chi^k(p)p^{-ks}.
$$
Moreover, exponentiating, using the Taylor series expansion for the exponential function and comparing coefficients shows that $b_p=a_p$ for all primes $p$. Hence, if $\Re s>1$, we may write
\begin{equation} \label{zerfall}
\log F(s,\chi)=M(s,\chi)+E(s,\chi), 
\end{equation}
where
$$
E(s,\chi):=\sum\limits_{p}\sum\limits_{k=2}^{\infty} b_{p^k}\chi^k(p)p^{-ks}.
$$

From condition d) on $F(s)$ in section {\rm \ref{assump}}, we deduce that $E(s,\chi)$
extends to a holomorphic function on the half plane $\Re s>1/2$, and
\begin{equation*} \label{errorest}
E\left(s,\chi\right)\ll \sum\limits_{p}\sum\limits_{k=2}^{\infty} b_{p^k}p^{-k\sigma} \ll \sum\limits_{p}\sum\limits_{k=2}^{\infty} p^{-k(\sigma-\varepsilon/2)}\ll_{\varepsilon} 1 \quad \mbox{if } \sigma\ge 1/2+\varepsilon.
\end{equation*}
This together with \eqref{zerfall} and conditions b) and c') on $F(s)$ implies that $M(s,\chi)$ extends to a holomorphic function on 
$$
\{s\ : \ \Re s> 1/2, \ \Im s\ge 1\}
$$
which satisfies 
\begin{equation} \label{Mesti}
M(s,\chi)\ll_{\varepsilon} \log(2qt)  \quad \mbox{uniformly for } \sigma \ge 1/2+\varepsilon \mbox{ and } t \ge 1.
\end{equation}

From \eqref{Hs}, \eqref{rewritten}, \eqref{umrechnung}, \eqref{Mesti}, $|\tau(\overline{\chi})|\le \sqrt{q}$ and $a_d\ll d^{\varepsilon}$, by the Ramanujan Hypothesis , we deduce that for every $\varepsilon>0$ and under the condition \eqref{T0condi}, $H(s)$ extends to a holomorphic function on 
\begin{equation}
\mathcal{M}_{\varepsilon}:=\{ s \ :\ \Re s\ge 1/2+\varepsilon,\ |\Im s|\le T_0\},
\end{equation}
and we have
\begin{equation} \label{Hsbound2}
H(s)\ll q^{1/2+\varepsilon}|T|^{\varepsilon} \quad \mbox{uniformly in } \mathcal{M}_{\varepsilon}.
\end{equation}

\section{Exponential integral estimates}
In this section we estimate the term $\mathcal{I}(s)$ introduced in \eqref{settings}. We write
\begin{equation} \label{intsplit}
\mathcal{I}(s)=\mathcal{I}_0(s)+\mathcal{I}_1(s)-\mathcal{I}_2(s),
\end{equation}
where 
\begin{equation} \label{splitting}
\begin{split}
\mathcal{I}_0(s):= & e(f(x_2)-g(x_2))\cdot (x_2^s-x_1^s),\\
\mathcal{I}_1(s):= & \int\limits_{x_1}^{x_2} 2\pi i (f'(u)-g'(u))u^{\sigma} \cdot e\left(f(u)-g(u)+\frac{t}{2\pi}\cdot \log u\right) \dif u,\\
\mathcal{I}_2(s):= & x_1^{s} \int\limits_{x_1}^{x_2} 2\pi i (f'(u)-g'(u)) \cdot e\left(f(u)-g(u)\right)\dif u.
\end{split}
\end{equation}
The first term is trivially bounded by
\begin{equation} \label{I0}
\mathcal{I}_0(s) \ll N^{\sigma}.
\end{equation}
For the estimation of $\mathcal{I}_1(s)$ and $\mathcal{I}_2(s)$, we need the following standard bound for exponential integrals.

\begin{lemma} \label{intbound}  Suppose that for some $k\ge 1$ and $\Lambda>0$ we have
$$
\left|f^{(k)}(x)\right|\ge \Lambda
$$
for any $x\in [a,b]$. Then
$$
\int\limits_a^b e(f(x))\dif x \ll_k \Lambda^{-1/k}.
$$
\end{lemma}

\begin{proof} This is Lemma 8.10. in \cite{IwKo}.
\end{proof}

Using integration by parts, we have 
\begin{equation} \label{intparts}
\begin{split}
\mathcal{I}_1(s)= & 2\pi i \left(f'(x_2)-g'(x_2)\right) x_2^{\sigma}\int\limits_{x_1}^{x_2} e\left(f(u)-g(u)+\frac{t}{2\pi}\cdot \log u\right) \dif u-\\
& 2\pi i \int\limits_{x_1}^{x_2} \left(\left(f''(u)-g''(u)\right)u^{\sigma}+\left(f'(u)-g'(u)\right)\sigma u^{\sigma-1}\right) \times\\ &  \int\limits_{x_1}^u 
e\left(f(v)-g(v)+\frac{t}{2\pi}\cdot \log v\right) \dif v \dif u. 
\end{split}
\end{equation}
From \eqref{approxi2} and Lemma \ref{intbound} with $k=3$, it follows that
\begin{equation} \label{Integral}
\int\limits_{x_1}^u 
e\left(f(v)-g(v)+\frac{t}{2\pi}\cdot \log v\right) \dif v \ll \frac{N}{f(N)^{1/3}} \quad \mbox{if } i=1,2, \ |t|\le T_0 \mbox{ and } x_1\le u\le x_2,
\end{equation}
provided that $T_0$ satisfies \eqref{T0condi}. Using \eqref{Qcond} and \eqref{approxi2}, we get
\begin{equation} \label{ableitungen}
\begin{split}
\left(f'(x_2)-g'(x_2)\right)x_2^{\sigma} \ll & \frac{N^{1+\sigma}}{(qQ)^2f(N)},\\
\left(f''(u)-g''(u)\right)u^{\sigma}+\left(f'(u)-g'(u)\right)\sigma u^{\sigma-1}\ll & \frac{1}{qQN^{1-\sigma}}\quad \mbox{if } x_1\le u\le x_2.
\end{split}
\end{equation}
From \eqref{m1m2sizes}, \eqref{intparts}, \eqref{Integral} and \eqref{ableitungen}, we deduce
\begin{equation} \label{I1}
\mathcal{I}_1(s)\ll \frac{N^{2+\sigma}}{(qQ)^2f(N)^{4/3}}.
\end{equation}
Similarly, we find that
\begin{equation} \label{I2}
\mathcal{I}_2(s)\ll \frac{N^{2+\sigma}}{(qQ)^2f(N)^{4/3}}.
\end{equation}

\section{Proofs of Theorems 1 and 2}
We prove Theorems 1 and 2 simultaneously, where we set $c_n=a_n$ for all $n\in \mathbb{N}$ in the case of Theorem 1 and  $c_n=a_n$ if $n$ prime and $c_n=0$ otherwise in the case of Theorem 2. We recall that so far, we have imposed the conditions \eqref{anotherfcond1}, \eqref{Qcond} and \eqref{T0condi}. Now we set
\begin{equation} \label{furtherTcond}
Q:= \frac{N^{1/2}}{f(N)^{1/3}}, \quad \sigma_0:=\frac{1}{2}+\varepsilon, \quad T_0:= T_1N^{-\eta} \quad \mbox{with} \quad T_1:=\min\left\{N,f(N)\right\}.
\end{equation}
The choice of $Q$ is consistent with \eqref{Qcond} if \eqref{anotherfcond1} holds. Under the choices in \eqref{furtherTcond},
the term on the right-hand side of \eqref{Otermbound} is bounded by
\begin{equation} \label{furtherOtermbound}
N^{\varepsilon}\left(1+\frac{N}{T_0}\right)\left(1+\frac{N^3}{q^3Q^3f(N)^2}\right)\ll \frac{N^{4+2\varepsilon}}{q^3Q^3f(N)^2T_1}
\end{equation}
if $q\le Q$ and $\eta<\varepsilon$.
Using \eqref{anotherfcond1}, \eqref{aftercauchy}, \eqref{Otermbound}, \eqref{Hsbound1}, \eqref{Hsbound2}, \eqref{intsplit}, \eqref{I0}, \eqref{I1}, \eqref{I2}, \eqref{furtherTcond} and \eqref{furtherOtermbound}, we get
\begin{equation} \label{conti}
\sum\limits_{x_1<n\le x_2} c_n e(f(n)) \ll \left(q^{1/2}N^{1/2}+\frac{N^{5/2}}{q^{3/2}Q^2f(N)^{4/3}} + \frac{N^{4}}{q^3Q^3f(N)^2T_1}\right)(qN)^{\varepsilon}.
\end{equation}

In section 3, we have divided the interval $[h(N'),h(N))$ into Farey intervals around fractions $l/q$ with
$$
1\le q\le Q, \quad l\asymp q\cdot h(N) \asymp q\cdot \frac{f(N)}{N} \quad \mbox{and} \quad (q,l)=1.  
$$
Hence, summing the contributions of the short sums in \eqref{conti} over all relevant $q$ and $l$, and using \eqref{anotherfcond1} and \eqref{furtherTcond}, we get
\begin{equation}
\begin{split}
\sum\limits_{n\sim N} c_n e(f(n)) \ll & (QN)^{\varepsilon} \sum\limits_{q\le Q}\ \sum\limits_{l\asymp qf(N)/N} \left(q^{1/2}N^{1/2}+
\frac{N^{5/2}}{q^{3/2}Q^2f(N)^{4/3}} + \frac{N^{4}}{q^3Q^3f(N)^2T_1}\right)\\ \ll
& (QN)^{\varepsilon} \left(\frac{Q^{5/2}f(N)}{N^{1/2}}+\frac{N^{3/2}}{Q^{3/2}f(N)^{1/3}}+\frac{N^3}{Q^3f(N)T_1}\right)\\
\ll & (QN)^{\varepsilon}\left(N^{3/4}f(N)^{1/6}+\frac{N^{3/2}}{T_1}\right) \ll N^{3/4+\varepsilon}f(N)^{1/6}.\quad \quad \Box
\end{split}
\end{equation}

\section{Proof of Theorem 3}
Throughout this section, we set $\gamma:=1/c$ and note that $11/12<\gamma<1$. Then $\left[n^c\right]=p$ is equivalent to
$$
-(p+1)^{\gamma}<-n\le -p^{\gamma}.
$$
Therefore, we have
$$
\sum\limits_{\substack{n\le N\\ \left[n^c\right]\in \mathbb{P}}} a_{\left[n^c\right]} = \sum\limits_{p\le N^c} \left(\left[-p^{\gamma}\right]-\left[-(p+1)^{\gamma}\right]\right)
a_p +O\left(N^{\varepsilon}\right),
$$
where we use the Ramanujan Hypothesis to bound the error term. It follows that
$$
\sum\limits_{\substack{n\le N\\ \left[n^c\right]\in \mathbb{P}}} a_{\left[n^c\right]} = \sum\limits_{p\le N^c} \left((p+1)^{\gamma}-p^{\gamma}\right)a_p+\sum\limits_{p\le N^c} \left(\psi\left(-(p+1)^{\gamma}\right)-\psi\left(-p^{\gamma}\right)\right)a_p+O\left(N^{\varepsilon}\right),
$$
where $\psi(x)$ is the saw tooth function, defined by
$$
\psi(x):=x-[x]-1. 
$$ 
Fourier analysing the function $\psi(x)$ as in \cite{baierzhao}, we see that to prove Theorem 3, it now suffices to establish that 
\begin{equation}
\sum\limits_{1\le |j|\le N^{1-\gamma+\eta}} \left| \sum\limits_{p\sim N} a_pe(jp^{\gamma}) \right| \ll N^{1-\delta}
\end{equation}
for some $\delta>0$, with $\eta$ arbitrarily small but fixed. 
We check that the function
\begin{equation} \label{concretef}
f(x)=jx^{\gamma}
\end{equation}
satisfies the conditions in section \ref{assump} if $1\le |j|\le N^{1-\gamma+\eta}$ and $\eta$ is small enough. From Theorem \ref{expsumprimes}, we deduce that
$$
\sum\limits_{1\le |j|\le N^{1-\gamma+\eta}} \left| \sum\limits_{p\sim N} a_pe(jp^{\gamma}) \right| \ll N^{23/12-\gamma+\varepsilon} \ll N^{1-\delta}
$$
for some $\delta>0$ if $\eta$ and $\varepsilon$ are small enough, proving the claim. $\Box$

\section{Proof of Theorems 5} 
It follows from Theorem 3 and Remark 1 that 
$$
\sum\limits_{\substack{n\le N\\ \left[n^c\right]\in \mathbb{P}}} \lambda\left(\left[n^c\right]^2\right) = \sum\limits_{p\le N^c} \left((p+1)^{1/c}-p^{1/c}\right)\lambda(p^2) + O\left(N^{1-\delta}\right)
\quad \mbox{as } N\rightarrow \infty
$$
for some $\delta>0$. Moreover, under the Riemann Hypothesis for $L(\mbox{Sym}^2\ G,s)$, the prime number theorem for this $L$-function takes the form
$$
\sum\limits_{p\le x} \lambda(p^2) \ll x^{1/2+\varepsilon}
$$
which implies that
$$
\sum\limits_{p\le N^c} \left((p+1)^{1/c}-p^{1/c}\right)\lambda(p^2) \ll N^{1-c/2+\varepsilon}
$$
using partial summation and $(p+1)^{1/c}-p^{1/c}\ll p^{1/c-1}$. Hence,
$$
\sum\limits_{\substack{n\le N\\ \left[n^c\right]\in \mathbb{P}}} \lambda\left(\left[n^c\right]^2\right)\ll N^{1-\delta}
$$
if $\delta>0$ is small enough. 

Now we use the well-known relation
$$
\lambda(p)^2=1+\lambda\left(p^2\right)
$$
to deduce that
\begin{equation} \label{sums}
\sum\limits_{\substack{n\le N\\ \left[n^c\right]\in \mathbb{P}}} \lambda\left(\left[n^c\right]\right)^2 = 
\sum\limits_{\substack{n\le N\\ \left[n^c\right]\in \mathbb{P}}} 1 + O\left(N^{1-\delta}\right). 
\end{equation}
The ordinary Piatetski-Shapiro prime number theorem (see \cite{baierzhao}, for example) tells us that
$$
\sum\limits_{\substack{n\le N\\ \left[n^c\right]\in \mathbb{P}}} 1 \sim \frac{N}{c\log N} \quad \mbox{as } N\rightarrow \infty
$$
for every fixed $c$ in the range in Theorem 3, completing the proof of Theorem 5. $\Box$

$ $\\
\noindent\begin{tabular}{p{8cm}p{8cm}}
Stephan Baier\\
University of East Anglia\\
Norwich Research Park\\
Norwich\\
NR4 7TJ\\
UK\\
Email: {\tt S.Baier@uea.ac.uk}
\end{tabular}
\end{document}